\documentclass{amsart}
\usepackage{amsthm}
\newtheorem{theorem}{Theorem}[section]
\newtheorem{lemma}[theorem]{Lemma}
\newtheorem{proposition}[theorem]{Proposition}
\theoremstyle{definition}
\newtheorem{definition}[theorem]{Definition}

\theoremstyle{remark}
\newtheorem{remark}[theorem]{Remark}

\usepackage{amssymb}
\usepackage{empheq}
\usepackage{stmaryrd}
\usepackage{enumerate}
\usepackage[shortlabels]{enumitem}
\usepackage{calc}
\usepackage{url}

\usepackage{xcolor}

\newcommand{\norm}[1]{\left\lVert#1\right\rVert}
\usepackage{mathtools}

\newcommand{\remin}{\mathop{-\!\!\!\!\!\hspace*{1mm}\raisebox{0.5mm}{$
\cdot$}}\nolimits}
\newcommand{\N}{{\mathbb N}}

\usepackage[left=2.0cm,%
                right=2.0cm,%
                top=2.5cm,%
                bottom=3.5cm,%
                headheight=12pt,%
                a4paper]{geometry}%

\begin{document}

\title[On the asymptotic behavior of second-order Cauchy problems]{Rates of convergence for the asymptotic behavior of second-order Cauchy problems}

\author[Nicholas Pischke]{Nicholas Pischke}
\date{\today}
\maketitle
\vspace*{-5mm}
\begin{center}
{\scriptsize Department of Mathematics, Technische Universit\"at Darmstadt,\\
Schlossgartenstra\ss{}e 7, 64289 Darmstadt, Germany, \ \\ 
E-mail: pischke@mathematik.tu-darmstadt.de}
\end{center}

\maketitle
\begin{abstract}
We provide a quantitative version of a result due to Poffald and Reich on the asymptotic behavior of solutions of a second-order Cauchy problem generated by an accretive operator in the form of a rate of convergence. This quantitative result is then used to generalize a result of Xu on the asymptotic behavior of almost-orbits of the solution semigroup of a first-order Cauchy problem to this second-order case.
\end{abstract}

\section{Introduction}

One of the fundamental questions in the theory of differential equations is that of the asymptotic behavior of the solutions to a particular system, e.g., for the well-studied (see \cite{Bar1976,Paz1983} for canonical references among many others) first-order system
\[
\begin{cases}
u'(t)\in -Au(t),\, 0<t<\infty\\
u(0)=x
\end{cases}
\tag{$\dagger$}
\]
over a Banach space $X$ generated by an initial value $x\in X$ and an accretive set-valued operator $A:X\to 2^X$. By the fundamental results of Brezis and Pazy \cite{BP1970} as well as Crandall and Liggett \cite{CL1971b}, the main tool used in the study of such a system is the semigroup $\mathcal{S}=\{S(t)\mid t\geq 0\}$ on $\overline{\mathrm{dom}A}$ generated by $A$ via the Crandall-Liggett exponential formula
\[
S(t)x=\lim_{n\to\infty}\left(\mathrm{Id}+\frac{t}{n}A\right)^{-n}x.
\]
In this paper, we are mainly concerned with this associated question of the asymptotic behavior of the corresponding generalized solutions $S(t)x$ for $t\to\infty$ in the context of uniformly convex and uniformly smooth spaces.\\

As is well-known, these semigroups do not converge asymptotically, even in the case of simple operators over Hilbert spaces and since the late 1970s, there has been a search for suitable conditions on both spaces and operators such that the convergence of the orbits of the solution semigroup can be guaranteed. One influential work in that context is that of Pazy \cite{Paz1978} where he introduced the so-called convergence condition for the operator $A$: over a Hilbert space $X$ with inner product $\langle\cdot,\cdot\rangle$, a maximally monotone operator $A$ is said to satisfy the convergence condition if for all bounded sequences $(x_n,y_n)\subseteq A$ with
\[
\lim_{n\to\infty}\langle y_n,x_n-Px_n\rangle=0,
\]
it holds that $\liminf_{n\to\infty}\norm{x_n-Px_n}=0$ where $P$ is the projection onto the closed and convex set $A^{-1}0$. Under that assumption, Pazy in \cite{Paz1978} showed the strong convergence of $S(t)x$ to a zero of the operator $A$ for $x\in\overline{\mathrm{dom}A}$ where $S$ is the semigroup generated by $A$ via the exponential formula as above. The convergence condition and the associated result on the asymptotic behavior of $S$ have subsequently been extended to the context of uniformly convex and uniformly smooth Banach spaces by Nevanlinna and Reich in \cite{NR1979} by modifying the premise to the assumption that 
\[
\lim_{n\to\infty}\langle y_n,J(x_n-Px_n)\rangle=0
\]
where $J$ is the normalized-duality map of $X$.\\

In this paper, we are concerned with a result due to Poffald and Reich \cite{PR1986} which generalizes the work of Nevanlinna and Reich to incomplete second-order Cauchy problems. Namely, for the second-order system 
\[
\begin{cases}
u''(t)\in Au(t),\;0<t<\infty,\\
u(0)=x,\\
\sup\{\norm{u(t)}\mid t\geq 0\}<\infty,
\end{cases}\tag{$\dagger$}
\]
over a uniformly smooth and uniformly convex Banach space $X$ with a strongly monotone duality map $J$ and $A$ m-accretive as before, the solution set 
\[
\mathcal{S}=\{u_x(t)\mid u_x\text{ is a solution almost everywhere with initial value }x\}
\]
is a nonlinear semigroup for $x\in\mathrm{dom}A$ as shown in \cite{PR1986}. Thus, this semigroup is generated via the exponential formula by some unique m-accretive operator which is denoted by $A_{1/2}$ and called the square root of $A$. Similarly, we write $\mathcal{S}_{1/2}$ for this semigroup. Various properties of this semigroup and the accompanying system were exhibited in \cite{PR1986}, generalizing previous work in the context of Hilbert spaces by Barbu \cite{Bar1972} as well as Brezis \cite{Bre1972}. In particular, Poffald and Reich obtained the following result on the asymptotic behavior of the semigroup:

\begin{theorem}[Poffald and Reich \cite{PR1986}]\label{thm:PR}
Let $X$ be uniformly convex and uniformly smooth with a strongly monotone duality map $J$ and $A$ be $m$-accretive with $A^{-1}0\neq\emptyset$ and such that it satisfies the convergence condition. If $\mathcal{S}_{1/2}=\{S_{1/2}(t)\mid t\geq 0\}$ is the semigroup generated by $A_{1/2}$ via the exponential formula as above, then $S_{1/2}(t)x$ converges strongly to a zero of $A$ for $t\to\infty$ for any $x\in\overline{\mathrm{dom}A}$.
\end{theorem}

In this paper, we exhibit the quantitative content of this result by extracting an explicit and computable transformation from the proof of Theorem \ref{thm:PR} which translates a so-called \emph{modulus of the convergence condition} (which was introduced in \cite{PP2022} to provide a quantitative representation of the way in which the operator in question satisfies the convergence condition), together with some minor quantitative data, into a full rate of convergence for the strong convergence of $S_{1/2}(t)x$ to a zero of $A$.\\

This result was established through the use of methods developed in proof mining, a program in mathematical logic which aims at the extraction of quantitative information from prima facie nonconstructive proofs. This proof mining program goes back conceptually to Kreisel's program of unwinding of proofs from the 1950's and, in its modern form, has been systematically developed since the 1990's by Ulrich Kohlenbach and his collaborators and by now comprises a large number of applications, in particular in nonlinear analysis and optimization (see \cite{Koh2008} for a book treatment and \cite{Koh2019} for a recent survey).

In that vein, this work can in particular be viewed as a new case study in this program for the theory of differential equations and abstract Cauchy problems, an area which so far has only seen applications by methods from proof mining in the early work by Kohlenbach and Koutsoukou-Argyraki \cite{KKA2015} and the work by Pinto and the author \cite{PP2022}.\\

Going beyond the range of proof mining however, we are here further concerned with new generalizations of the theorem of Poffald and Reich. In \cite{Xu2001}, Xu studied the behavior of almost-orbits associated with the semigroup generated by $A$ as introduced by Miyadera and Kobayasi \cite{MK1982}: an almost-orbit of $\mathcal{S}$ is a continuous function $u:[0,\infty)\to\overline{\mathrm{dom}A}$ such that
\[
\lim_{s\to\infty}\sup\{\norm{u(t+s)-S(t)u(s)}\mid t\geq 0\}=0.
\]
Concretely, Xu obtained the following result which generalizes the result of Nevanlinna and Reich to almost-orbits in the sense of the above:

\begin{theorem}[Xu \cite{Xu2001}]\label{thm:Xu}
Let $X$ be uniformly convex and uniformly smooth and $A$ be $m$-accretive with $A^{-1}0\neq\emptyset$ and such that it satisfies the convergence condition. If $\mathcal{S}=\{S(t)\mid t\geq 0\}$ is the semigroup generated by $A$ via the exponential formula, then every almost-orbit $u(t)$ of $\mathcal{S}$ converges strongly as $t\to\infty$.
\end{theorem}

By combining the ideas of the quantitative analysis obtained in \cite{PP2022} of this result for the first-order case together with the quantitative version of the result of Poffald and Reich established in the first part of this paper, we here obtain a quantitative version of a result on almost-orbit convergence for the semigroup $\mathcal{S}_{1/2}$. This result, while finitary in nature, in particular also implies back the following ``infinitary" result for $\mathcal{S}_{1/2}$ which is similar to Xu's result above:

\begin{theorem}\label{thm:XuGen}
Let $X$ be uniformly convex and uniformly smooth with a strongly monotone duality map $J$ with value $M>0$, i.e. $\langle x-y,Jx-Jy\rangle\geq M\norm{x-y}^2$ for all $x,y \in X$. Let $A$ be m-accretive such that it satisfies the convergence condition and that $A^{-1}0\neq\emptyset$. Let $\mathcal{S}_{1/2}=\{S_{1/2}(t) : t\geq 0\}$ be the semigroup generated by $A_{1/2}$ via the exponential formula. Then every almost-orbit $u(t)$ of $\mathcal{S}_{1/2}$ converges strongly as $t\to\infty$.
\end{theorem}

This result on the behavior of almost-orbits in the case of $\mathcal{S}_{1/2}$ seems to be new to the literature and the approach taken here to establish it in particular exhibits the strength of quantitative analyses obtained in the proof mining program as these, exhibiting the real finitary core of a mathematical argument, sometimes allow for easy generalizations that lead to new results.\\

However, as common in applications of the proof mining program, the results given here are presented without any use of logical tools.

\section{Preliminaries}

As in the context of the work of Poffald and Reich \cite{PR1986}, we throughout consider a Banach space $(X, \norm{\cdot})$ with dual space $X^*$ which is uniformly convex, i.e.
\[
\forall \varepsilon\in (0,2]\exists \delta\in (0,1]\forall x,y\in X\left (\norm{x},\norm{y}\leq 1\land \norm{x-y}\geq\varepsilon\to \norm{\frac{x+y}{2}}\leq 1-\delta\right),
\]
and uniformly smooth, i.e.
\[
\forall \varepsilon>0\exists\delta>0\forall x,y\in X\left( \norm{x}=1\land \norm{y}\leq\delta\to \norm{x+y}+\norm{x-y}\leq 2+\varepsilon\norm{y}\right).
\]
Associated with $X$ is the normalized duality mapping $J:X\to X^*$ , defined by
\[
J(x):=\left\{ x^*\in X^*\mid \langle x,x^*\rangle=\norm{x}^2\text{ and }\norm{x^*}=\norm{x}\right\},
\]
for all $x\in X$ and this mapping is single-valued and uniformly continuous if, and only if, $X$ is uniformly smooth (see [8]). 

Further, we assume that we are given an accretive set-valued operator $A:X\to 2^X$, i.e.
\[
\forall (x_1,y_1),(x_2,y_2)\in A\left(\langle y_1-y_2,J(x_1-x_2)\rangle\geq 0\right),
\]
which is further m-accretive, i.e. $\mathrm{ran}(Id+\gamma A) = X$ for all $\gamma > 0$. We write $\mathrm{dom}(A) := \left\{x\in X \mid Ax\neq\emptyset\right\}$ for the domain and $\mathrm{ran}(A) :=\bigcup_{x\in X} Ax$ for the range of $A$.

Lastly, we henceforth write $P$ for the nearest point projection onto the set $A^{-1}0$ which will be assumed to be non-empty and therefore this projection is well-defined as the space is uniformly convex.

\section{The convergence condition from a quantitative perspective}

As discussed in the introduction, the central notion for the asymptotic results of Poffald and Reich is the notion of the convergence condition for the operator $A$ inducing the differential equation.

\begin{definition}[Nevanlinna and Reich \cite{NR1979}, generalizing Pazy \cite{Paz1978}]
Let $A$ be an m-accretive operator $A$ on a uniformly convex and uniformly smooth space $X$ with (single-valued) duality map $J$ and let $A^{-1}0\neq\emptyset$ and $P$ be the projection onto $A^{-1}0$. Then $A$ is said to satisfy the convergence condition if for all bounded sequences $(x_n,y_n)\subseteq A$ with
\[
\lim_{n\to\infty}\langle y_n,J(x_n-Px_n)\rangle=0,
\]
it holds that $\liminf_{n\to\infty}\norm{x_n-Px_n}=0$.
\end{definition}

In the analysis of Theorem \ref{thm:PR}, we will rely on a particular quantitative representation of that condition introduced in \cite{PP2022} based on logical considerations on different formal versions of the convergence condition (see \cite{PP2022} for details). Concretely, this quantitative representation is given by the following modulus:\footnote{These moduli are called \emph{full} moduli in \cite{PP2022}.}

\begin{definition}[\cite{PP2022}]
A \emph{modulus for the convergence condition} of an operator $A$ is a function $\Omega:\mathbb{N}\times\mathbb{N}\to\mathbb{N}$ satisfying that for all $k,K\in\mathbb{N}$ and all $x,y\in X$: if $y\in Ax$ and $\norm{x},\norm{y}\leq K$, then
\[
\langle y,J(x-Px)\rangle\leq\frac{1}{\Omega(k,K)+1}\Rightarrow\norm{x-Px}\leq\frac{1}{k+1}.
\]
\end{definition}

Clearly, any operator $A$ which posses such a modulus satisfies the convergence condition but, in and of itself, the requirement that $\Omega$ translates \emph{local errors of} $\langle y,J(x-Px)\rangle$ into \emph{local errors of} $\norm{x-Px}$ in such a uniform manner seems to be potentially stronger than that of the convergence condition. However, the following result was established in \cite{PP2022}:

\begin{proposition}[\cite{PP2022}]\label{pro:CCchar}
An operator $A$ satisfies the convergence condition if, and only if, it posses a modulus for the convergence condition.
\end{proposition} 

So: the convergence condition, although formulated via sequences, is indeed a uniform transformation of local errors in the sense of the previous..\\

Regarding further explorations of the naturalness of such moduli, we refer to \cite{PP2022} for various explicit constructions of such moduli for certain classes of operators which do satisfy the convergence condition (in particular covering strongly accretive operators and operators which are uniformly accretive at zero in the sense of \cite{KKA2015}).\\

Before we move to the quantitative treatment of Theorem \ref{thm:PR}, we collect some pointers for interesting logical features of the above modulus in the following remark.

\begin{remark}[For logicians]\label{rem:logic}
The above modulus is the natural quantitative reformulation of the convergence condition as guided by the monotone functional interpretation together with the negative translation and in particular as established in \cite{PP2022}, general logical metatheorems established in \cite{Pis2022b,Pis2022} (relying on the previous seminal works \cite{GeK2008,Koh2005}) guarantee both 
\begin{enumerate}
\item the extractability of a computable full modulus for the convergence condition from a wide range of noneffective proofs of the already much weaker property
\[
\forall (x,y)\in A\left( \langle y,J(x-Px)\rangle=0\to \norm{x-Px}=0\right),
\]
\item that from a noneffective proof using the convergence condition as a premise, a transformation can be extracted that transforms a full modulus into quantitative information on the conclusion.
\end{enumerate}
We refer to \cite{PP2022} for sketches of proofs for these and further logical remarks.
\end{remark}

\section{An analysis of Poffald's and Reich's result}

To derive a quantitative version of the convergence result contained in Theorem \ref{thm:PR}, depending on a modulus for the convergence condition, we first have to extract from the proof of \cite{PR1986} explicit quantitative bounds on the norms of the orbits and their derivatives involved.\\

For that, we follow the way a solution for the associated system ($\dagger$) is constructed in \cite{PR1986} (which differs in comparison to the construction of Barbu \cite{Bar1972} (see also \cite{Bre1972}) who considered this problem in the context of Hilbert spaces before Poffald and Reich). To solve ($\dagger$), Poffald and Reich first solve the system
\[
\begin{cases}
u''(t)\in Au(t)+pu(t),\; 0<t<\infty,\\
u(0)=x,\\
u\in L^2(0,\infty;X).
\end{cases}\tag*{$(\dagger)_p$}
\]
for $p\to 0^+$ which in turn is solved by studying the approximate system
\[
\begin{cases}
u''(t)=A_ru(t)+pu(t),\; 0<t<\infty,\\
u(0)=x,\\
u\in L^2(0,\infty;X).
\end{cases}\tag*{$(\dagger)^r_p$}
\]
for $r\to 0^+$ where $A_r$ is the Yosida approximate.\\

In the latter case, they conclude that the unique solution $u^r_p$ of $(\dagger)_p^r$ converges in $L^2(0,\infty;X)$ and $C([0,\infty);X)$ to a (unique) solution $u_p$ of $(\dagger)_p$. For the approximate solutions $u^r_p$, the following bounds on $u^r_p$ and its derivatives are obtained:
\begin{itemize}
\item $\norm{u^r_p(t)}\leq\norm{x}$ for all $t\geq 0$ (p. 521, (2.7));
\item $\int_0^\infty\norm{{u^r_p}'(t)}^2\,\mathrm{dt}\leq 2/M^2(d(0,Ax)+p\norm{x})^{3/2}\norm{x}^{1/2}$ (p. 522, (2.14));
\item $\int_0^\infty\norm{{u^r_p}''(t)}^2\,\mathrm{dt}\leq 2/M^2(d(0,Ax)+p\norm{x})^{1/2}\norm{x}^{3/2}$ (p. 522, (2.17)).
\end{itemize}
As remarked in \cite{PR1986}, these bounds immediately transfer to the solution $u_p$ of $(\dagger)_p$ by applying Lemma 2.6 of \cite{PR1986} to $u^r_p\to u_p$ for $r\to 0^+$.\\

Following \cite{PR1986}, these bounds can then be used to establish bound on the respective norms of a solution $u$ to $(\dagger)$ by applying Lemma 2.7 of \cite{PR1986} to the convergence $u_p\to u$ for $p\to 0^+$ which immediately yields the following bounds for the solution of $(\dagger)$ corresponding to the initial value $x$:
\begin{itemize}
\item $\norm{u(t)}\leq\norm{x}$ for all $t\geq 0$;
\item $\int_0^\infty\norm{u'(t)}^2\,\mathrm{dt}\leq 2/M^2d(0,Ax)^{3/2}\norm{x}^{1/2}$;
\item $\int_0^\infty\norm{u''(t)}^2\,\mathrm{dt}\leq 2/M^2d(0,Ax)^{1/2}\norm{x}^{3/2}$.
\end{itemize}

Besides of these bounds, we will make use of the following lemma:
\begin{lemma}[folklore, see \cite{PP2022}]\label{lem:integralLimInf}
If $f:[0,\infty)\to[0,\infty)$ is Lebesgue integrable with
\[
\int^{\infty}_0f(t)\,\mathrm{dt}\leq L,
\]
then for any Lebesgue null set $N\subseteq [0,\infty)$ and any $k,n$:
\[
\exists t\in [n,(L+1)(k+1)+n]\setminus N\left(f(t)\leq\frac{1}{k+1}\right).
\]
\end{lemma}

The quantitative version of Theorem \ref{thm:PR} now takes the following form for the case of $x\in\mathrm{dom}A$.

\begin{theorem}\label{thm:PRquant}
Let $X$ be uniformly convex and uniformly smooth with a strongly monotone duality map $J$ with value $M>0$, i.e. $\langle x-y,Jx-Jy\rangle\geq M\norm{x-y}^2$ for all $x,y \in X$. Let $A$ be $m$-accretive with $A^{-1}0\neq\emptyset$ with $p\in A^{-1}0$ and such that it satisfies the convergence condition with a modulus for the convergence condition $\Omega$. Let $\mathcal{S}_{1/2}=\{S_{1/2}(t)\mid t\geq 0\}$ be the semigroup generated by $A_{1/2}$ via the exponential formula. For any $x\in\mathrm{dom}A$, we have
\[
\forall k\in\mathbb{N}\forall t,t'\geq\chi((\Omega(2k+1,\max\{1,b\})+1)^2\remin 1)\left(\norm{S_{1/2}(t)x-S_{1/2}(t')x}\leq \frac{1}{k+1}\right)
\]
with $\chi(k)=(D+1)(k+1)$ and where 
\[
D\geq(1+b^2)\frac{2}{M^2}d(0,Ax)^{1/2}b^{3/2}
\]
as well as $b\geq\norm{x-Px},\norm{x}$.
\end{theorem}
\begin{proof}
We write $u(t)=S_{1/2}(t)x$. Then $u''$ exists almost everywhere, say on $[0,\infty)$. As outline in the discussion before Lemma \ref{lem:integralLimInf}, we have $\norm{u(t)}\leq\norm{x}$ for all $t\geq 0$ as well as
\[
\int_0^\infty\norm{u''(t)}^2\,\mathrm{dt}\leq\frac{2}{M^2}d(0,Ax)^{1/2}\norm{x}^{3/2}.
\]
Now, using the defining property of the projection $P$, we have
\[
\norm{u(t+h)-Pu(t+h)}\leq\norm{u(t+h)-Pu(t)}\leq\norm{u(t)-Pu(t)}
\]
which in particular implies that
\begin{align*}
\int_0^\infty\langle u''(t),J(u(t)-Pu(t))\rangle^2\,\mathrm{dt}&\leq\int_0^\infty\norm{u''(t)}^2\norm{J(u(t)-Pu(t))}^2\,\mathrm{dt}\\
&\leq\int_0^\infty\norm{u''(t)}^2\norm{J(u(0)-Pu(0))}^2\,\mathrm{dt}\\
&\leq \frac{2}{M^2}d(0,Ax)^{1/2}\norm{x}^{3/2}\norm{x-Px}^2.
\end{align*}
Therefore also
\begin{align*}
&\int_0^\infty\left(\norm{u''(t)}^2+\langle u''(t),J(u(t)-Pu(t))\rangle^2\right)\,\mathrm{dt}\\
&\qquad\qquad\qquad\leq (1+\norm{x-Px}^2)\frac{2}{M^2}d(0,Ax)^{1/2}\norm{x}^{3/2}\\
&\qquad\qquad\qquad\leq D.
\end{align*}
Lemma \ref{lem:integralLimInf} now implies that for any $k\in\mathbb{N}$:
\[
\exists t\in [0,\chi(k)]\setminus N\left(\max\left\{\norm{u''(t)}^2,\langle u''(t),J(u(t)-Pu(t))\rangle^2\right\}\leq\frac{1}{k+1}\right).
\]
Thus in particular, we have
\[
\exists t\in [0,\chi((k+1)^2\remin 1)]\setminus N\left(\max\left\{\norm{u''(t)},\langle u''(t),J(u(t)-Pu(t))\rangle\right\}\leq\frac{1}{k+1}\right)
\]
which yields
\begin{gather*}
\exists t\leq\chi((\Omega(k,\max\{1,b\})+1)^2\remin 1)\\\left(\max\left\{\norm{u''(t)},\langle u''(t),J(u(t)-Pu(t))\rangle\right\}\leq\frac{1}{\Omega(k,\max\{1,b\})+1}\right)
\end{gather*}
and thus, as $\norm{u''(t)}\leq 1$ for such a $t$, the properties of $\Omega$ yield that
\[
\exists t\leq\chi((\Omega(k,\max\{1,b\})+1)^2\remin 1)\left(\norm{u(t)-Pu(t)}\leq\frac{1}{k+1}\right).
\]
But as discussed above, $\norm{u(t)-Pu(t)}$ is decreasing and thus actually
\[
\forall t\geq\chi((\Omega(k,\max\{1,b\})+1)^2\remin 1)\left(\norm{u(t)-Pu(t)}\leq\frac{1}{k+1}\right).
\]
As in \cite{PR1986}, we can now show
\[
\norm{u(t+h)-u(t)}\leq 2\norm{u(t)-Pu(t)}
\]
and thus we obtain
\[
\forall t\geq\chi((\Omega(2k+1,\max\{1,b\})+1)^2\remin 1)\forall h\left(\norm{u(t+h)-u(t)}\leq\frac{1}{k+1}\right)
\]
which is the claim.
\end{proof}

By continuity of $S_{1/2}$, the result for $x\in\mathrm{dom}A$ extends to $x\in\overline{\mathrm{dom}A}$ and by an analysis of this proof, we obtain the following quantitative result for the extension.

\begin{theorem}
Assume the conditions of Theorem \ref{thm:PRquant}. Let $x\in\overline{\mathrm{dom}A}$ with $f:\mathbb{N}\to\mathbb{N}$ be such that $f$ is nondecreasing and 
\[
\forall k\in\mathbb{N}\exists z,y\in X\left (z\in Ay\land \norm{y},\norm{z}\leq f(k)\land \norm{x-y}\leq\frac{1}{k+1}\right).
\]
Then
\begin{gather*}
\forall k\in\mathbb{N}\forall t,t'\geq \chi^+_k((\Omega(6k+5,\max\{1,f(3k+2)\})+1)^2\remin 1)\\
\left(\norm{S_{1/2}(t)x-S_{1/2}(t')x}\leq\frac{1}{k+1} \right)
\end{gather*}
with $\chi_k(k)=(D_k+1)(k+1)$ and where 
\[
D_k\geq(1+b_k^2)\frac{2}{M^2}f(3k+2)^{2}
\]
as well as $b_k\geq \norm{x-Px}+\norm{x}+f(3k+2)$.
\end{theorem}
\begin{proof}
Using that $S_{1/2}(t)$ is a contraction for every $t$, we get that there exists $z\in Ay$ such that $\norm{z},\norm{y}\leq f(3k+2)$ and $\norm{x-y}\leq 1/(3k+3)$. Therefore
\begin{align*}
\norm{S_{1/2}(t)x-S_{1/2}(t')x}&\leq\norm{S_{1/2}(t)x-S_{1/2}(t')y}+\norm{S_{1/2}(t)y-S_{1/2}(t')y}\\
&\qquad\qquad\qquad+\norm{S_{1/2}(t')x-S_{1/2}(t')y}\\
&\leq 2\norm{x-y} + \norm{S_{1/2}(t)y-S_{1/2}(t')y}\\
&\leq \frac{2}{3(k+1)} + \norm{S_{1/2}(t)y-S_{1/2}(t')y}.
\end{align*}
Using the previous Theorem \ref{thm:PRquant}, we get that 
\begin{gather*}
\forall k\in\mathbb{N}\forall t,t'\geq\chi_k((\Omega(6k+5,\max\{1,f(3k+2)\})+1)^2\remin 1)\\\left(\norm{S_{1/2}(t)y-S_{1/2}(t')y}\leq \frac{1}{3k+3}\right)
\end{gather*}
since
\begin{align*}
\norm{y-Py}&\leq\norm{x-Px}+\left\vert\norm{y-Py}-\norm{x-Px}\right\vert\\
&\leq\norm{x-Px}+\norm{y-x}\\
&\leq\norm{x-Px}+\norm{x}+f(3k+2)
\end{align*}
and thus $\norm{y-Py}\leq b_k$. This gives the claim.
\end{proof}

\section{A generalization to almost-orbits}

We are now concerned with establishing Theorem \ref{thm:XuGen}. As discussed in the introduction, we arrive at this generalization of Xu's result by means of an initial quantitative result and we arrive at this quantitative result in turn by utilizing the above quantitative version of Poffald's and Reich's result together with the ideas employed in the analysis of Xu's result given in \cite{PP2022}.\\

In that context, the (logically speaking) complicated premise of $u$ being an almost-orbit induces two natural quantitative versions of that property which were introduced in \cite{KKA2015} and also feature in the finitary variants of Xu's result given in \cite{PP2022}. Concretely, in the following, we will obtain (similar to \cite{KKA2015,PP2022}) two translations converting respectively
\begin{enumerate}
\item a rate of metastability $\Phi$ of the almost-orbit as introduced in \cite{KKA2015}, i.e. $\Phi$ satisfies
\[
\forall k\in\N\ \forall f:\N\to\N\ \exists n\leq \Phi(k,f)\ \forall t\in [0,f(n)] \left( \|S_{1/2}(t)u(n)-u(t+n)\|\leq \frac{1}{k+1} \right),
\]
into a rate of metastability $\Gamma$ for the Cauchy property of the almost-orbit, i.e. $\Gamma$ satisfies
\[
\forall k\in\N\ \forall f:\N\to \N\ \exists n\leq \Gamma(k, f)\ \forall t,t'\in [n, n+f(n)] \left( \|u(t)-u(t')\|\leq \frac{1}{k+1} \right),
\]
\item a rate of convergence $\Phi$ for the almost-orbit, i.e. $\Phi$ satisfies
\[
\forall k\in\mathbb{N}\ \forall s\geq \Phi(k)\left( \sup_{t\geq 0}\norm{u(s+t)-S_{1/2}(t)u(s)}\leq\frac{1}{k+1}\right).
\]
into a rate of Cauchyness of the almost-orbit of the Cauchy problem in a similar manner as before.
\end{enumerate}

The need for metastability in the context of quantitative results on convergence statements in general, and in proof mining in particular, arises from fundamental results from recursion theory due to Specker \cite{Spe1949} whereas even computable monotone sequences of rational numbers in $[0,1]$ do not have a computable rate of convergence. By generalizing such examples, one can see that also in general a rate of convergence for an almost-orbit will not be computable (see \cite{KKA2015}). Even if computable rates of convergence are in general unattainable, one can, in very general situations, provide effective rates of so-called metastability (which also has been recognized as an important finitary version of the Cauchy property from a non-logical perspective by Tao, see e.g. \cite{Tao2008b,Tao2008a}) which are, moreover, highly uniform. In particular, metastability is (noneffectively) equivalent to the Cauchy property. We refer to \cite{PP2022} as well as \cite{KKA2015} for further (logical) discussions on these different quantitative versions of the almost-orbit property.\\

We now begin with the metastable version of the generalization of which both Theorem \ref{thm:XuGen} and a second quantitative result on rates of convergence will be corollaries:

\begin{theorem}\label{thm:Xustylequantmeta}
Let $X$ be uniformly convex and uniformly smooth with a strongly monotone duality map $J$ with value $M>0$, i.e. $\langle x-y,Jx-Jy\rangle\geq M\norm{x-y}^2$ for all $x,y \in X$. Let $A$ be m-accretive such that it satisfies the convergence condition with a modulus for the convergence condition $\Omega$. Let $\mathcal{S}_{1/2}=\{S_{1/2}(t) : t\geq 0\}$ be the semigroup generated by $A_{1/2}$ via the exponential formula. Let $A^{-1}0\neq\emptyset$ with $p\in A^{-1}0$ and assume that $P$, the nearest point projection onto $A^{-1}0$, is uniformly continuous on bounded subsets of $X$ with a modulus $\omega:\mathbb{N}^2\to\mathbb{N}$ such that
\[
\forall r,k\in\N\ \forall x,y\in B_r(p) \left( \|x-y\|\leq \frac{1}{\omega(r,k)+1} \to \|Px-Py\|\leq \frac{1}{k+1} \right),
\]
and, without loss of generality, assume that $\omega(r,k)\geq k$ for all $r,k\in\mathbb{N}$. Let $u$ be an almost-orbit of $\mathcal{S}_{1/2}$ with a rate of metastability $\Phi$ on the almost-orbit condition, i.e.
\[
\forall k\in\N\ \forall f:\N\to\N\ \exists n\leq \Phi(k,f)\ \forall t\in [0,f(n)] \left( \|S_{1/2}(t)u(n)-u(t+n)\|\leq \frac{1}{k+1} \right).
\]
Let $B\in\mathbb{N}^*$ be such that $\|u(t)-p\|\leq B$ for all $t\geq 0$ and let $f_s:\mathbb{N}\to\mathbb{N}$ for $s\geq 0$ be such that $f_s$ is nondecreasing and
\begin{align*}
&\forall n\in\mathbb{N}\ \exists x_{s,n},y_{s,n}\in X\bigg(y_{s,n}\in Ax_{s,n}\\
&\qquad\qquad\qquad\land \norm{x_{s,n}},\norm{y_{s,n}}\leq f_s(n)\land \norm{x_{s,n}-u(s)}\leq\frac{1}{n+1}\bigg).
\end{align*}
Then we have
\[
\forall k\in\N\ \forall f:\N\to \N\ \exists n\leq \Gamma(k, f)\ \forall t,t'\in [n, n+f(n)] \left( \|u(t)-u(t')\|\leq \frac{1}{k+1} \right),
\]
where
\[
\Gamma(k, f):=\max\{\Gamma'(8k+7,j_{k,f}),\Phi(8k+7, h_{N,f})\mid N\leq \Gamma'(8k+7, j_{k,f})\}
\]
with
\begin{gather*}
h_{N,f}(n):=f(\max\{N,n \}) + \max\{N,n\} - n,\\
j_{k,f}(n):=\max\{n,\Phi(8k+7,h_{n,f})\}-n\\
g_{k,f}(m):=\Omega_{m}(3k+2) + f(m+\Omega_{m}(3k+2)),\\
\Gamma'(k,f):= \Phi(\omega(B, 3k+2), g_{k,f})+\max\{\Omega_{m}(3k+2)\mid m\leq\Phi(\omega(B, 3k+2), g_{k,f})\},
\end{gather*}
for $\Omega_s(k)$ with $s\geq 0$ defined by
\[
\Omega_s(k):=\chi_{s,k}((\Omega(3k+2,\max\{1,f_s(\omega(B+1, 3k+2))\})+1)^2\remin 1)
\]
with $\chi_{s,k}(k):=(D_{s,k}+1)(k+1)$ and where 
\[
D_{s,k}\geq(1+(B+1)^2)\frac{2}{M^2}f_s(\omega(B+1, 3k+2))^2.
\]
\end{theorem}

Although the proof is, in essence, a careful reimplementation of the proof for the analogous quantitative result for Xu's theorem established in \cite{PP2022}, we nevertheless present the following proof in a self-contained manner.

\begin{proof}
For a given $s\geq 0$, note that by definition of $f_s$ there exist $y_{s,k}\in Ax_{s,k}$ with $\norm{x_{s,k}},\norm{y_{s,k}}\leq f_s(\omega(B+1, 3k+2))$ such that
\[
\|x_{s,k}-u(s)\|\leq \frac{1}{\omega(B+1, 3k+2)+1}\ \left(\leq \frac{1}{3(k+1)}\right).
\]
Since
\[
\|x_{s,k}-Px_{s,k}\|\leq \|x_{s,k}-p\|\leq \|x_{s,k}-u(s)\|+\|u(s)-p\|\leq B+1,
\]
we have as in the proof of Theorem \ref{thm:PRquant} that
\[
\forall k\in\N\ \forall t\geq\Omega_s(k) \left( \|S_{1/2}(t)x_{s,k}-PS_{1/2}(t)x_{s,k}\| \leq \frac{1}{3(k+1)}\right),
\]
with $\Omega_s(k)$ defined as above. This then yields
\begin{align*}
&\|S_{1/2}(t)u(s)-PS_{1/2}(t)u(s)\|\\
&\qquad\qquad\leq \norm{x_{s,k}-u(s)}+\norm{S_{1/2}(t)x_{s,k}-PS_{1/2}(t)x_{s,k}}\\
&\qquad\qquad\qquad\qquad +\|PS_{1/2}(t)x_{s,k}-PS_{1/2}(t)u(s)\|\\
&\qquad\qquad\leq \frac{1}{3(k+1)} + \frac{1}{3(k+1)} + \|PS_{1/2}(t)x_{s,k}-PS_{1/2}(t)u(s)\|
\end{align*}
which, as $\|S_{1/2}(t)x_{s,k}-S_{1/2}(t)u(s)\|\leq \|x_{s,k}-u(s)\|\leq 1/(\omega(B+1, 3k+2)+1)$, then yields
\[
\forall s\geq 0\ \forall k\in \N\ \forall t\geq \Omega_s(k) \left( \|S_{1/2}(t)u(s)-PS_{1/2}(t)u(s)\| \leq \frac{1}{k+1}\right).\tag{1}\label{eq:1}
\]

\medskip

\noindent For given $k\in\N$ and $f:\N\to \N$, we now consider the function $g_{k,f}$ as defined above. Using the assumption on $\Phi$, there is some $n_0\leq \Phi(\omega(B, 3k+2), g_{k,f})$ such that
\[
\forall t \in [0, g_{k,f}(n_0)] \left( \|S_{1/2}(t)u(n_0)-u(t+n_0)\|\leq \frac{1}{\omega(B, 3k+2)+1} \right).
\]
Since $\|S_{1/2}(t)u(n_0)-p\|, \|u(t+n_0)-p\|\leq B$, we conclude that
\[
\forall t \in [0, g_{k,f}(n_0)] \left( \|PS_{1/2}(t)u(n_0)-Pu(t+n_0)\|\leq \frac{1}{3(k+1)} \right).
\]
Thus, for $t\in[0, g_{k,f}(n_0)]$, we get by a simple triangle inequality that
\[
\|u(t+n_0)-Pu(t+n_0)\|\leq \frac{2}{3(k+1)} + \|S_{1/2}(t)u(n_0)-PS_{1/2}(t)u(n_0)\|.
\]
Using (\ref{eq:1}), this yields
\[
\forall t\in [n_0+\Omega_{n_0}(3k+2), n_0+g_{k,f}(n_0)] \left( \|u(t)-Pu(t)\|\leq \frac{1}{k+1} \right)
\]
which yields that
\[
\forall k\in\N,f:\N\to\N\ \exists n\leq \Gamma'(k,f)\ \forall t\in [n, n+f(n)] \left( \|u(t)-Pu(t)\|\leq \frac{1}{k+1} \right).\tag{2}\label{eq:2}
\]

\medskip

\noindent By assumption on $\Phi$, there is $n_0\leq \Phi(2k+1, h_{N,f})$ such that
\[
\forall t\leq h_{N,f}(n_0) \left( \|S_{1/2}(t)u(n)-u(t+n)\|\leq \frac{1}{2(k+1)} \right),
\]
with $h_{N,f}(n)$ defined as above. Writing $n:=\max\{N,n_0\}\in [N, \max\{N, \Phi(2k+1, h_{N,f})\}]$, we have
\[
\|S_{1/2}(t)u(n)-u(t+n)\|\leq \|u(n)-S_{1/2}(n-n_0)u(n_0)\|+ \|S_{1/2}(t+n-n_0)u(n_0)-u(t+n)\|
\]
using simple triangle inequalities and noting that $n-n_0\leq t+n-n_0\leq h_{N,f}(n_0)$, we conclude that
\begin{gather*}
\forall k,N\in\N,f:\N\to\N\ \exists n\in[N,\max\{N,\Phi(2k+1,h_{N,f})\}]\\
\forall t\leq f(n) \left( \|S_{1/2}(t)u(n)-u(t+n)\|\leq \frac{1}{k+1} \right).\tag{3}\label{eq:3}
\end{gather*}
	
\medskip

\noindent Let now $k\in\N$ and $f:\N\to\N$ be given. From (\ref{eq:2}) with the function $j_{k,f}(n)$ defined as above, there is an $n_0\leq \Gamma'(8k+7, j_{k,f})$ such that
\[
\forall t\in [n_0, n_0+j_{k,f}(n_0)] \left( \|u(t)-Pu(t)\|\leq \frac{1}{8(k+1)} \right).
\]
By (\ref{eq:3}), there exists $n_1\in[n_0, \max\{n_0, \Phi(8k+7, h_{n_0,f})\}]$ satisfying
\[
\forall t\leq f(n_1) \left( \|S_{1/2}(t)u(n_1)-u(t+n_1)\|\leq \frac{1}{4(k+1)} \right).
\]
Since $n_1\in [n_0,\max\{n_0, \Phi(8k+7, h_{n_0,f})\}]=[n_0, n_0+j_{k,f}(n_0)]$, we also have $\|u(n_1)-Pu(n_1)\|\leq 1/(8(k+1))$. Thus, by simple triangle inequalities, we get for any $t\leq f(n_1)$:
\[
\|u(n_1)-u(t+n_1)\|\leq\frac{1}{2(k+1)}.
\]
Noting additionally that
\[
n_1\leq \max\{n_0, \Phi(8k+7, h_{n_0,f})\}
\]
yields that 
\begin{gather*}
\forall k\in\N\ \forall f:\N\to \N\ \exists n_0\leq \Gamma'(8k+7, j_{k,f}), n_1\leq \max\{n_0, \Phi(8k+7,h_{n_0,f})\}\\
\forall t\leq f(n_1) \left( \|u(n_1)-u(t+n_1)\|\leq \frac{1}{2(k+1)} \right).\tag{4}\label{eq:4}
\end{gather*}
	
\medskip
	
\noindent Lastly, by another simple triangle inequality it follows that
\[
\forall t, t'\in [n_1, n_1+f(n_1)] \left( \|u(t)-u(t')\|\leq \frac{1}{k+1} \right)
\]
for the $n_1$ chosen by (\ref{eq:4}) which yields
\[
\forall k\in\N\ \forall f:\N\to \N\ \exists n\leq \Gamma(k, f)\ \forall t,t'\in [n, n+f(n)] \left( \|u(t)-u(t')\|\leq \frac{1}{k+1} \right)
\]
since 
\begin{align*}
n_1&\leq\max\{n_0, \Phi(8k+7, h_{n_0,f})\}\\
&\leq \max\{ \Gamma'(8k+7, j_{k,f}),\Phi(8k+7, h_{N,f})\}\mid N\leq \Gamma'(8k+7, j_{k,f}) \} \\
&= \Gamma(k,f).
\end{align*}
\end{proof}

This finitary result now in particular implies a usual ``infinitary" result on the convergence of almost-orbits of $\mathcal{S}_{1/2}$ as formulated in Theorem \ref{thm:XuGen} since metastability trivially (though non-effectively) implies back convergence of the respective sequence.

\begin{proof}[Proof of Theorem \ref{thm:XuGen}]
Let $X$ be uniformly convex and uniformly smooth with a strongly monotone duality map $J$ with value $M>0$ and let $A$ be m-accretive such that it satisfies the convergence condition and that $A^{-1}0\neq\emptyset$. Let $u$ be an almost-orbit of $\mathcal{S}_{1/2}$. By Proposition \ref{pro:CCchar}, there exists a modulus for the convergence condition $\Omega$. As in \cite{KKA2015}, it is rather immediate to see that $u$ has a rate of metastability $\Phi$. Then, after naturally fixing the other minor quantitative data as required in the above theorem (which naturally exist), we get by Theorem \ref{thm:Xustylequantmeta} that there exists a function $\Gamma$ such that
\[
\forall k\in\N\ \forall f:\N\to \N\ \exists n\leq \Gamma(k, f)\ \forall t,t'\in [n, n+f(n)] \left( \|u(t)-u(t')\|\leq \frac{1}{k+1} \right).
\]
In particular, by forgetting about $\Gamma$, we simply have
\[
\forall k\in\N\ \forall f:\N\to \N\ \exists n\in\mathbb{N}\ \forall t,t'\in [n, n+f(n)] \left( \|u(t)-u(t')\|\leq \frac{1}{k+1} \right)
\]
and this formulation implies the Cauchy property of $u(t)$ as follows: suppose $u(t)$ is not Cauchy for $t\to\infty$, i.e.
\[
\exists k\in\N\ \forall n\in\mathbb{N}\ \exists t,t'\geq n \left( \|u(t)-u(t')\|> \frac{1}{k+1} \right)
\]
and define $f(n)$ non-effectively such that $f(n)+n\geq t,t'$ for these two $t,t'$ guaranteed by this property. Then for that $k$ and $f$:
\[
\forall n\in\mathbb{N}\ \exists t,t'\in [n, n+f(n)] \left( \|u(t)-u(t')\|> \frac{1}{k+1} \right)
\]
which is in contradiction to the metastability of $u$.
\end{proof}

Lastly, similar to both \cite{KKA2015} and \cite{PP2022}, we can also give a second quantitative version of Theorem \ref{thm:XuGen} based on the previously discussed strengthened premise of a rate of convergence for the almost-orbit. This then takes the form of the following theorem.

\begin{theorem}\label{thm:Xuquantroc}
Let $X$ be uniformly convex and uniformly smooth with a strongly monotone duality map $J$ with value $M>0$, i.e. $\langle x-y,Jx-Jy\rangle\geq M\norm{x-y}^2$ for all $x,y \in X$. Let $A$ be $m$-accretive such that there exists a weak modulus for the convergence condition $\Omega$. Let $\mathcal{S}_{1/2}=\{S_{1/2}(t)\mid t\geq 0\}$ is the semigroup generated by $A$ via the exponential formula. Let $A^{-1}0\neq\emptyset$ with $p\in A^{-1}0$ and assume that $P$, the nearest point projection onto $A^{-1}0$, is uniformly continuous on bounded subsets of $X$ with a modulus $\omega:\mathbb{N}^2\to\mathbb{N}$, i.e. 
\[
\forall r,k\in\N\ \forall x,y\in B_r(p) \left( \|x-y\|\leq \frac{1}{\omega(r,k)+1} \to \|Px-Py\|\leq \frac{1}{k+1} \right),
\]
and, without loss of generality, assume that $\omega(r,k)\geq k$ for all $r,k\in\mathbb{N}$. Let $u$ be an almost orbit with a rate of convergence $\Phi:\mathbb{N}\to\mathbb{N}$ on the almost-orbit condition, i.e.
\[
\forall k\in\mathbb{N}\ \forall s\geq \Phi(k)\left( \sup_{t\geq 0}\norm{u(s+t)-S_{1/2}(t)u(s)}\leq\frac{1}{k+1}\right).
\]
Let $B\in\mathbb{N}^*$ be such that $\|u(t)-p\|\leq B$ for all $t\geq 0$ and let $f_s:\mathbb{N}\to\mathbb{N}$ for $s\geq 0$ be such that $f$ is nondecreasing and
\begin{align*}
&\forall n\in\mathbb{N}\ \exists x_{s,n},y_{s,n}\in X\bigg(y_{s,n}\in Ax_{s,n}\\
&\qquad\qquad\qquad\land \norm{x_{s,n}},\norm{y_{s,n}}\leq f_s(n)\land \norm{x_{s,n}-u(s)}\leq\frac{1}{n+1}\bigg).
\end{align*}
Then we have
\[
\forall k\ \forall t,t'\geq \max\{\Phi(8k+7),s^*+\Omega_{s^*}(24k+23)\}\left(\norm{u(t)-u(t')}\leq\frac{1}{k+1}\right)
\]
where $s^*=\Phi(\omega(B,24k+23))$ and where $\Omega_s(k)$ is defined as in Theorem \ref{thm:Xustylequantmeta}.
\end{theorem}
\begin{proof}
Given a rate of convergence $\Phi$ on the almost-orbit condition, it is easy to check that the function $\Phi(k,f):=\Phi(k)$ is a rate of metastability for the almost-orbit in the previous sense. Theorem \ref{thm:Xustylequantmeta} thus yields that $\Gamma(k,f)$ as constructed therein is rate of metastability for $u(t)$. By the results of \cite{KPin2022}, Proposition 2.6, a function $\rho:(0,\infty)\to (0,\infty)$ is a Cauchy rate of a sequence iff $\varphi(\varepsilon,f):=\rho(\varepsilon)$ is a rate of metastability. By inspecting the bound $\Gamma(k,f)$ and noting that $\Phi(k,f)=\Phi(k)$, we can easily see that this independence on $f$ transfers to $\Gamma$. Thus, by Proposition 2.6 of \cite{KPin2022} we get that $\Gamma(k):=\Gamma(k,f)$ is a rate of convergence and this yields the bound given above by simplifying the expressions given in Theorem \ref{thm:Xustylequantmeta} to define $\Gamma$ accordingly.
\end{proof}

\noindent
{\bf Acknowledgment:}  I want to thank Ulrich Kohlenbach and Pedro Pinto for valuable comments on this paper. 

The author was supported by the
`Deutsche Forschungs\-gemein\-schaft' Project DFG KO 1737/6-2.

\bibliographystyle{plain}
\bibliography{ref}

\begin{thebibliography}{10}

\bibitem{Bar1972}
V.~Barbu.
\newblock {A class of boundary problems for second order abstract differential
  equations}.
\newblock {\em Journal of the Faculty of Science, University of Tokyo},
  19:295--319, 1972.

\bibitem{Bar1976}
V.~Barbu.
\newblock {\em {Nonlinear semigroups and differential equations in Banach
  spaces}}.
\newblock Springer Netherlands, 1976.

\bibitem{Bre1972}
H.~Brezis.
\newblock {Equations d’evolution du second ordre associees a des operateurs
  monotones}.
\newblock {\em Israel Journal of Mathematics}, 12:51--60, 1972.

\bibitem{BP1970}
H.~Brezis and A.~Pazy.
\newblock {Accretive sets and differential equations in Banach spaces}.
\newblock {\em Israel Journal of Mathematics}, 8:367--383, 1970.

\bibitem{CL1971b}
G.~Crandall and T.M. Liggett.
\newblock {Generation of semigroups of nonlinear transformations on general
  Banach spaces}.
\newblock {\em American Journal of Mathematics}, 93:265--298, 1971.

\bibitem{GeK2008}
P.~Gerhardy and U.~Kohlenbach.
\newblock {General logical metatheorems for functional analysis}.
\newblock {\em Transactions of the American Mathematical Society},
  360:2615--2660, 2008.

\bibitem{Koh2005}
U.~Kohlenbach.
\newblock {Some logical metatheorems with applications in functional analysis}.
\newblock {\em Transactions of the American Mathematical Society},
  357(1):89--128, 2005.

\bibitem{Koh2008}
U.~Kohlenbach.
\newblock {\em {Applied Proof Theory: Proof Interpretations and their Use in
  Mathematics}}.
\newblock Springer Monographs in Mathematics. Springer-Verlag Berlin
  Heidelberg, 2008.

\bibitem{Koh2019}
U.~Kohlenbach.
\newblock {Proof-theoretic Methods in Nonlinear Analysis}.
\newblock In B.~Sirakov, P.~Ney de~Souza, and M.~Viana, editors, {\em Proc. ICM
  2018}, volume~2, pages 61--82. World Scientific, 2019.

\bibitem{KKA2015}
U.~Kohlenbach and A.~Koutsoukou-Argyraki.
\newblock {Rates of convergence and metastability for abstract Cauchy problems
  generated by accretive operators}.
\newblock {\em Journal of Mathematical Analysis and Applications},
  423:1089--1112, 2015.

\bibitem{KPin2022}
U.~Kohlenbach and P.~Pinto.
\newblock {Quantitative translations for viscosity approximation methods in
  hyperbolic spaces}.
\newblock {\em Journal of Mathematical Analysis and Applications}, 507:125823,
  33 pages, 2022.

\bibitem{MK1982}
I.~Miyadera and K.~Kobayasi.
\newblock {On the asymptotic behaviour of almost-orbits of nonlinear
  contraction semigroups in Banach spaces}.
\newblock {\em Nonlinear Analysis: Theory, Methods \& Applications},
  6(4):349--365, 1982.

\bibitem{NR1979}
O.~Nevanlinna and S.~Reich.
\newblock {Strong convergence of contraction semigroups and of iterative
  methods for accretive operators in Banach spaces}.
\newblock {\em Israel Journal of Mathematics}, 32:44--58, 1979.

\bibitem{Paz1978}
A.~Pazy.
\newblock {Strong convergence of semigroups on nonlinear contractions in
  Hilbert space}.
\newblock {\em Journal of Mathematical Analysis and Applications}, 34:1--35,
  1978.

\bibitem{Paz1983}
A.~Pazy.
\newblock {\em {Semigroups of Linear Operators and Applications to Partial
  Differential Equations}}.
\newblock Applied Mathematical Sciences. Springer New York, NY, 1983.

\bibitem{PP2022}
P.~Pinto and N.~Pischke.
\newblock {On computational properties of Cauchy problems generated by
  accretive operators}.
\newblock {\em ArXiv e-prints}, 2023.
\newblock arXiv, math.AP, 2301.06880.

\bibitem{Pis2022b}
N.~Pischke.
\newblock {A proof-theoretic metatheorem for nonlinear semigroups generated by
  an accretive operator and applications}.
\newblock {\em ArXiv e-prints}, 2023.
\newblock arXiv, math.LO, 2304.01723.

\bibitem{Pis2022}
N.~Pischke.
\newblock {Logical metatheorems for accretive and (generalized) monotone
  set-valued operators}.
\newblock 2023.
\newblock To appear in Journal of Mathematical Logic.

\bibitem{PR1986}
E.I. Poffald and S.~Reich.
\newblock {An incomplete Cauchy problem}.
\newblock {\em Journal of Mathematical Analysis and Applications},
  113(2):514--543, 1986.

\bibitem{Spe1949}
E.~Specker.
\newblock {Nicht konstruktiv beweisbare S\"atze der Analysis}.
\newblock {\em Journal of Symbolic Logic}, 14:145--158, 1949.

\bibitem{Tao2008b}
T.~Tao.
\newblock {Norm Convergence of Multiple Ergodic Averages for Commuting
  Transformations}.
\newblock {\em Ergodic Theory and Dynamical Systems}, 28:657--688, 2008.

\bibitem{Tao2008a}
T.~Tao.
\newblock {\em Structure and Randomness: Pages from Year One of a Mathematical
  Blog}, chapter {Soft analysis, hard analysis, and the finite convergence
  principle}.
\newblock American Mathematical Society, Providence, RI, 2008.

\bibitem{Xu2001}
H.-K. Xu.
\newblock {Strong asymptotic behavior of almost-orbits of nonlinear
  semigroups}.
\newblock {\em Nonlinear Analysis: Theory, Methods \& Applications},
  46(1):135--151, 2001.

\end{thebibliography}

\end{document}